\documentclass[12pt]{article}
\usepackage[utf8]{inputenc} %for Windows
\usepackage[german, english]{babel}
\usepackage[tbtags]{amsmath}
\usepackage{amsfonts,amssymb,mathrsfs,amscd,comment}
\usepackage{tikz}
\usepackage{amsthm}
\usepackage[a4paper, total={6in, 9in}]{geometry}

\usepackage{fancyhdr}
\DeclareMathOperator{\mmod}{mod}

\DeclareMathOperator{\Pol}{R}
\DeclareMathOperator{\Log}{G}

\DeclareMathOperator{\Li}{Li}

\newtheorem{lemma}{Lemma}
\newtheorem*{lemma*}{Lemma}
\newtheorem{theorem}{Theorem}
\newtheorem*{theorem*}{Theorem}

\newtheorem{proposition}{Proposition}

\newtheorem*{remark*}{Remark}
\newtheorem{corollary} {Corollary}
\newtheorem*{corollary*} {Corollary}

\newcommand*{\bfrac}[2]{\genfrac{}{}{0pt}{}{#1}{#2}}
{}{}{}

\global\long\def\veps{\varepsilon}

\title{A logarithmic improvement in the Bombieri-Vinogradov theorem}
\date{}
\author{Alisa Sedunova}

\begin{document}

\maketitle

\begin{abstract}
In this paper we improve the best known to date result of \cite{DIT}, getting $(\log x)^2$ instead of $(\log x)^{\frac{5}{2}}$. We use a weighted form of Vaughan's identity, allowing a smooth truncation inside the procedure, and an estimate due to Barban-Vehov \cite{Barban2009} and Graham \cite{Michigan1978} related to Selberg's sieve. We give effective and non-effective versions of the result. From that one can derive the fully effective Bombieri-Vinogradov theorem for $q \leq x^{\frac{1}{2}-\veps}$. The ineffectivity is avoided by applying an effective result by Landau and Page for small moduli $q$ instead using Siegel-Walfisz theorem. 
\end{abstract}

\section{Introduction}
For integer number $a$ and $q \geq 1$, let 
$$
%\psi(x)=\sum_{\substack{n\leq x }} \Lambda(n) \qquad \text{ and } \qquad 
\psi(x; q, a)=\sum_{\substack{n\leq x \\ {n\equiv a (\mmod{q})}}} \Lambda(n),
$$
where $\Lambda(n)$ is the von Mangoldt function. 
The Bombieri-Vinogradov theorem is an estimate for the error terms in the prime number theorem for arithmetic progressions averaged over all $q$ up to $x^{1/2}$, or, rather almost all $q$ up to $x^{\frac{1}{2}}$. 
\begin{theorem*}[Bombieri-Vinogradov]
Let $A$ be a given positive number and $Q\leq \frac{x^{1/2}}{(\log{x})^{B}},$ where
$B=B(A)$. Then
$$\sum_{q \leq Q} \max_{2 \leq y \leq x} {\max_{\bfrac{a}{(a,q)=1}}} \left|\psi(y,q,a) - \frac{y}{\varphi(q)} \right| \ll_A \frac{x}{(\log{x})^A}.$$
\end{theorem*}
The implied constant in this theorem is not effective, since we have to take care of characters associated with those $q$ that have small prime factors. At the same time, effective versions - in which the effect of an exceptional character is avoided in one way or another - have been known since \cite{Len2016} and \cite{Timofeev}, and, very recently, \cite{Liu2017}.
We state the main result of this paper.

\begin{theorem} [Bombieri-Vinogradov, ineffective]\label{bvcor}
Let $A$ be a positive number and $Q \leq \frac{x^{1/2}}{(\log x)^A}$. Then we have the following bound 
\begin{equation*}
\begin{split}
\sum_{q \leq Q} \max_{2 \leq y \leq x} {\max_{\bfrac{a}{(a,q)=1}}} \left|\psi(y,q,a) - \frac{y}{\varphi(q)}  \right| \ll_A \frac{x}{(\log x)^{A-2}}.
\end{split}
\end{equation*}
\end{theorem}
The implied constant in Theorem \ref{bvcor} is ineffective. We give an effective version of the result above together with its applications in Section \ref{effectivity}.

Previously, the best result of the type of Theorem \ref{bvcor} in the literature followed from \cite{DIT}; it had $A - 5/2$ instead of $A - 2$. While \cite{DIT} does not state the result in full -- focusing on estimating a crucial sum -- a complete form can be found in \cite{bvecsimple} (together with a fully explicit version). It is
$$
\sum_{\bfrac{q\leq Q }{ l(q)>Q_1}} \max_{2 \leq y \leq x} \max_{\bfrac{a}{(a,q)=1}} \left| \psi(y;q,a) - \frac{\psi(y)}{\varphi(q)} \right| \leq C \left(\frac{x}{Q_1}+x^{\frac{1}{2}}Q+x^{\frac{2}{3}}Q^{\frac{1}{2}}+x^{\frac{5}{6}}\log \frac{Q}{Q_1}\right)(\log x)^{\frac{7}{2}},$$
where $C$ is an explicit absolute constant (a similar fully explicit result was proven in \cite{Akbary2015} with $(\log x)^{\frac{9}{2}}$ instead of $(\log x)^{\frac{7}{2}}$).
Another effective variant without explicit
constants is given by Lenstra and Pomerance \cite[Lemma 11.2]{Len2016} (with bigger power of $\log$) in their work on Gaussian periods.

\begin{comment}
The result of Theorem \ref{bveffective} is completely new and is a modification of the proof of Theorem \ref{bvcor}. The key new ingredients are the large sieve with prime support, Landau-Page theorem and an application of results by Ramar\'{e} \cite{Ramare2014a}.
\end{comment}

\begin{remark*}
Define
$$
\pi(x) = \sum_{p \leq x} 1 \qquad \text{ and } \qquad \pi(x;q,a) = \sum_{\bfrac{p \leq x}{ p \equiv a (\mmod{q}) }} 1.
$$
For $x \geq 4$, $1 \leq Q_1 \leq Q \leq x^{\frac{1}{2}}$ and any $\veps >0$ we have
$$
\sum_{\bfrac{q\leq Q }{ l(q)>Q_1}} \max_{2 \leq y \leq x} \max_{\bfrac{a}{(a,q)=1}} \left| \pi(y;q,a) - \frac{\pi(y)}{\varphi(q)} \right| 
\ll x^{\frac{1}{2}}Q(\log x)^2+\frac{x}{Q_1}(\log x)^3 + x^{\frac{13}{14}+\veps} (\log x)^4.$$
\end{remark*}
The proof of the remark is exactly the same as in \cite{Akbary2015}, we just have to change the power of $\log x$.

The key tool for the proof of Theorem \ref{bombierivinogradovimproved} is Vaughan's identity, which we have to get in an explicit version for our goal.
Define 
$$\psi(y,\chi)= \sum_{n\leq y} \Lambda(n) \chi(n),$$
the twisted summatory function for the von Mangoldt function $\Lambda$ and a Dirichlet character $\chi$ modulo $q$.
The key tool in getting Theorem \ref{bvcor} is the following estimate.
\begin{proposition}[Vaughan's inequality, improved] \label{vaug2}
For $x \geq 4$ and any $\veps > 0$ we have
\begin{equation*}
\sum_{q \leq Q} \frac{q}{\varphi(q)} \sideset{}{^*}\sum_{\chi (q)} \max_{y \leq x} |\psi(y,\chi)| \ll \left(x+Q^2x^{\frac{1}{2}}+Q x^{\frac{13}{14}+\veps}\right)(\log x)^{2},
\end{equation*}
where $Q$ is any positive real number and $\sideset{}{^*}\sum_{\chi (q)}$ means a sum over all primitive characters $\chi (\mmod q)$.
\end{proposition}
The improvement here consists in having a factor of $(\log x)^2$, rather than $(\log x)^{\frac{5}{2}}$ or $(\log x)^3$. 
In order to prove Proposition \ref{vaug2} we use the weighted version of Vaughan's identity (see Lemma \ref{vaughanweighted}) and an estimate due to Barban-Vehov \cite{Barban2009} and Graham \cite{Michigan1978}. While Graham uses the Siegel-Walfisz theorem, there is an effective (and explicit) version of it in \cite{Andres2014}. We follow methods developed in \cite{Andres2014} in the proof.

Proposition \ref{vaug2} allows us to prove the Bombieri-Vinogradov theorem in the form of Theorem \ref{bvcor} and, hence, Corollary \ref{bombierivinogradovimproved}. In addition to Theorem \ref{bombierivinogradovimproved}, the proof uses the Siegel-Walfisz theorem, which states that
$$\psi(x,\chi) - \delta(\chi)x \ll_Ax e^{-c\sqrt{\log x}}$$
uniformly for $q \leq (\log x)^A$. Here $A > 0$ is a fixed real number, $c$ is an absolute positive constant, and $\delta(\chi) =1$ if $\chi$ is principal and is zero otherwise. The implied constant
in the Bombieri-Vinogradov theorem is ineffective since the implied constant in the Siegel-Walfisz theorem is ineffective. To prove Corollary \ref{bvcor} we use the Siegel-Walfisz
theorem to deal with moduli $q \leq Q$ having small prime divisors and Theorem \ref{bombierivinogradovimproved} to deal with the sum over the remaining moduli. 

\subsection{Effectivity} \label{effectivity}
We formulate the corollary of the main result.
\begin{corollary} [Bombieri-Vinogradov, with exceptional character taken out]
\label{bombierivinogradovimproved}
Let $x \geq 4$, $1 \leq Q_1 \leq Q \leq x^{\frac{1}{2}}$. Denote by $l(q)$ the smallest prime divisor of $q$.
Then for any positive $\veps > 0$ we have
$$
\sum_{\bfrac{q\leq Q }{ l(q)>Q_1}} \max_{2 \leq y \leq x} \max_{\bfrac{a}{(a,q)=1}} \left| \psi(y;q,a) - \frac{\psi(y)}{\varphi(q)} \right| 
\ll x^{\frac{1}{2}}Q(\log x)^2+\frac{x}{Q_1}(\log x)^3 + x^{\frac{13}{14}+\veps} (\log x)^4.$$
\end{corollary}
The implied constant is effective and can be made explicit using \cite{Andres2014} together with the best available constant in P\'{o}lya-Vinogradov inequality given in \cite{Frolenkov2013}. The effectivity is attained by getting rid of those moduli that have small prime divisors, thus of a possible exceptional character.

The recent work of Liu \cite{Liu2017} gives us a genunely effective Bombiei-Vinogradov theorem. This is ultimately due to the fact that we can use an effective Landau-Page result (see \cite{Page}, \cite{Landau1918} and also \cite[Chapter 10]{Vinogradov}), which is non-trivial up to $(\log x)^2$ instead of making a standard ineffective step on applying Siegel-Walfisz theorem.

\begin{theorem*}[Liu, 2017]
There exists an effective positive constant $B$ such that
\[
\sum_{q \leq x^{1/2}/(\log x)^B} \max_{y \leq x} \max_{(a,q)=1} \left| \psi(y;q,a)-\frac{\Li y}{\varphi(q)}\right| \ll x \frac{(\log \log x)^9}{\log x}. 
\]
\end{theorem*}

\begin{comment}
\begin{theorem} [Bombieri-Vinogradov, effective]
\label{bveffective}
Let $x \geq 4$, $1 \leq Q_1 \leq Q \leq x^{\frac{1}{2}-\delta}$, $\delta>0$. 
For any positive $\veps < \frac{1}{28}$ we have
$$
\sum_{q\leq Q^{1-\delta}} \max_{2 \leq y \leq x} \max_{\bfrac{a}{(a,q)=1}} \left| \psi(y;q,a) - \frac{\psi(y)}{\varphi(q)} \right| 
\ll_{\delta} x^{\frac{1}{2}}Q \frac{(\log x)^2}{(\log \log x)^2},$$
where the implied constant is effective.
\end{theorem}
\end{comment}
In \cite{Liu2017} various applications of the statement above are considered, such as an asymptotic formula for the representation of a large integer as the sum of two squares and a prime and Titchmarsh divisor problem (both results obviously become effective).

\begin{comment}
With softening the condition $q \leq Q$ to $q \leq Q^{1-\delta}$, $\delta >0$ we get an improved version of Proposition \ref{vaug2}.
\begin{proposition} \label{vaug2foref}
For any positive $\veps<\frac{1}{28}$, $\delta >0$ and $x \geq 4$ we have
\begin{equation*}
\sum_{q \leq Q^{1-\delta}} \frac{q}{\varphi(q)} \sideset{}{^*}\sum_{\chi (q)} \max_{y \leq x} |\psi(y,\chi)| \ll_{\delta}Q^2x^{\frac{1}{2}} \frac{(\log x)^2}{(\log \log x)^2},
\end{equation*}
where $Q$ is any positive real number and $\sideset{}{^*}\sum_{\chi (q)}$ means a sum over all primitive characters $\chi (\mmod q)$.
\end{proposition}
To prove the Proposition \ref{vaug2foref} we use large sieve with prime support and get a cancellation in type I sums by applying results of \cite{Ramare2014a}.
Getting a smaller logarithmic factor in Proposition \ref{vaug2foref} is crucial for obtaining a genuinely effective Bombieri-Vinogradov theorem (Theorem \ref{bveffective}).
\end{comment}

\subsection*{Acknowledgements}
The author is grateful to Henryk Iwaniec for his crucial advice. The author also thanks her former supervisor Harald Helfgott for his help.

\section{Proof of Theorem \ref{bvcor}}
\subsection*{Auxiliary lemmas}
We start with a so-called weighted Vaughan identity. It allows us to get cancellation in type II sums.
\begin{lemma}[Weighted Vaughan identity] \label{vaughanweighted}
Let $U,V \geq 1$, $n>U$. Define a function $\eta(t): \mathbb{Z}^{+} \to \mathbb{R}$ with $\eta(t)=1$ for $t \leq V$.
We have
$$\Lambda(n)=\lambda_0(n)+\lambda_1(n)+\lambda_2(n)+\lambda_3(n),$$
where $\lambda_0(n)=\Lambda(n)$ for $n \leq U$ and equals to $0$ for $n \leq U$, and
\begin{equation*}
\begin{split}
&\lambda_1(n)=\sum_{d|n} \mu(d) \eta(d) \log \frac{n}{d},
\;\;\;\;\;\;\;\;\;\;\;\;\;\; 
\lambda_2(n)=-\sum_{c \leq U} \sum_{dc|n} \mu(d) \Lambda(c) \eta(d),\\
&\lambda_3(n)=\sum_{c > U} \sum_{dc|n} \mu(d) \Lambda(c) (1-\eta(d)).\\
\end{split}
\end{equation*}
\end{lemma}

\begin{proof}
Let $n > U$, since otherwise the statement is trivial.
Define the following quantities
\begin{equation*}
\begin{split}
&\Lambda_1(n) = \sum_{\bfrac{d|n}{d \leq V}} \mu(d) \log \frac{n}{d} = \lambda_1(n) - \sum_{\bfrac{d|n}{d>V}} \mu(d) \eta(d) \log \frac{n}{d} = \lambda_1(n)+\lambda_1'(n),\\
&\Lambda_2(n) = - \sum_{c \leq U} \sum_{\bfrac{dc|n}{d \leq V}} \mu(d) \Lambda(c) = \lambda_2(n) + \sum_{c \leq U}\sum_{\bfrac{dc|n}{d>V}} \mu(d) \Lambda(c) \eta(d) = \lambda_2(n)+\lambda_2'(n),\\
&\Lambda_3(n) = \sum_{c > U} \sum_{\bfrac{dc|n}{d > V}} \mu(d) \Lambda(c) = \lambda_3(n) + \sum_{c > U}\sum_{\bfrac{dc|n}{d>V}} \mu(d) \Lambda(c) \eta(d) = \lambda_3(n)+\lambda_3'(n).
\end{split}
\end{equation*}
Vaughan's identity in its classical form is 
$$\Lambda(n)=\Lambda_1(n)+\Lambda_2(n)+\Lambda_3(n),$$
so it remains to show that $\lambda_1'(n)+\lambda_2'(n)+\lambda_3'(n) = 0$ for every $n$. 
Let us rewrite this sum
\begin{equation*}
\begin{split}
\sum_{i=1}^{3}\lambda_i'(n) &= \sum_{\bfrac{d|n}{d>V}} \left(-\mu(d) \eta(d)\log \frac{n}{d}  + 
\sum_{\bfrac{c | \frac{n}{d}}{c \leq U}} \mu(d) \Lambda(c)\eta(d)+
\sum_{\bfrac{c | \frac{n}{d}}{c > U}} \mu(d) \Lambda(c) \eta(d) \right)\\
&=\sum_{\bfrac{d|n}{d>V}} \left( -\mu(d) \eta(d) \log \frac{n}{d}  + \mu(d) \eta(d) \sum_{c | \frac{n}{d}} \Lambda(c) \right) =0,\\
\end{split}
\end{equation*}
where in the last equality we used the fact that $\sum_{x|y} \Lambda(x) = \log y$.
\end{proof}

\begin{lemma}[Graham \cite{Michigan1978}] \label{graham1}
Let $1 \leq N_1 \leq N_2 \leq N$ and define
$$f_i(d) = \begin{cases}
\mu(d)\log \frac{N_i}{d}, &d \leq N_i,\\
0, &d>N_i.\\
\end{cases}
$$
We have
$$\sum_{n=1}^{N} \left(\sum_{d_1|n}f_1(d_1)\right)\left(\sum_{d_2|n}f_2(d_2)\right) = N \log N_1 + O(N).$$
\end{lemma}
From the lemma above one can deduce
\begin{corollary} \label{graham2}
Define a function $\eta(t)$, that is equal to $1$ for $t \leq V$, to $0$ for $t > V_0$ and
$$\eta(t) = \frac{\log \frac{V_0}{t}}{\log \frac{V_0}{V}}, \;\; V < t \leq V_0.$$ 
Then
$$\sum_{k \leq Y} \left|\sum_{d|k} \mu(d) \eta(d)\right|^2 \ll \frac{Y}{\log \frac{V_0}{V}}.$$
\end{corollary}
The constant here can be made explicit using \cite{Andres2014}.

We also need the large sieve inequality as stated in a classical form in, for example \cite[p.561]{Montgomery1978},
\begin{align} \label{sieve0}
\sum_{q \leq Q} \frac{q}{\varphi(q)} {\sideset{}{^*}\sum_{\chi (q)}} \left| \sum_{m = m_0 + 1}^{m_0 + M} a_m \chi(m) \right|^2
\leq (M + Q^2) \sum_{m = m_0 + 1}^{m_0 + M} |a_m|^2,
\end{align}
from which it follows that

\begin{lemma}[Large sieve inequality] \label{sieve}
Let $a_m$, $b_n$ be arbitrary complex numbers. 
Then
\begin{equation*}
\begin{split}
&\sum_{q \leq Q} \frac{q}{\varphi(q)} {\sideset{}{^*}\sum_{\chi (q)}} \max_{y} \left| \sum_{m = m_0}^{M} \sum_{\bfrac{n = n_0}{mn \leq y}}^{N} a_m b_n \chi(mn) \right| \leq \\
& c_3 (M' + Q^2)^{\frac{1}{2}}(N' + Q^2)^{\frac{1}{2}} \left(\sum_{m = m_0}^{M} |a_m|^2\right)^{\frac{1}{2}}\left(\sum_{n = n_0}^{N} |b_n|^2\right)^{\frac{1}{2}} L(M,N),
\end{split}
\end{equation*}
where 
$c_3=2.64...$, $L(M,N)=\log (2MN)$ and $M'=M-m_0+1$, $N'=N-n_0+1$ are the number of terms in the sums over $m$ and $n$ respectively.
\end{lemma}
For the proof see \cite[Lemma 6.1]{Akbary2015}. 

\subsection*{Proof of Proposition \ref{vaug2}}
We proceed now with the proof of Proposition \ref{vaug2}. Fix arbitrary real numbers $Q > 0$ and $x \geq 4$.
Without loss of generality we can assume that $2 \leq Q \leq x^{1/2}$ and decompose the von Mangoldt function using a weighted form of Vaughan's identity, namely Lemma \ref{vaughanweighted}.
$$
\Lambda(n) = \lambda_0(n) + \lambda_1(n) + \lambda_2(n) + \lambda_3(n),
$$
where $\lambda_i(n)$, $i = 0,1,2,3$ are as in the statement of the lemma and $U, V, V_0$ are parameters. Notice also that we are free to choose $\eta(t)$ as we wish, we only need to fulfill the conditions stated in Lemma \ref{vaughanweighted}. 

Assume $y \leq x$, $q \leq Q$, and $\chi$ is a character mod $q$.  We use the above decomposition to write
$$
\psi(y,\chi) = s_0+ s_1 + s_2 + s_3,
$$
where
$$
s_i = \sum_{n \leq y} \lambda_i(n) \chi(n).
$$
Denote the contributions to our main sum by
$$S_i = \sum_{q \leq Q} \frac{q}{\varphi(q)} \sideset{}{^*}\sum_{\chi (q)} \max_{y \leq x} |s_i|.$$
Easily we obtain
$$\sum_{q \leq Q} \frac{q}{\varphi(q)} \sideset{}{^*}\sum_{\chi (q)} \max_{y \leq x} |\psi(y,\chi)| \leq S_0+ S_1+S_2+S_3,$$
where 
\begin{equation}
\begin{split}
&S_0 \ll U Q^2,\\
&s_1 = \sum_{d \leq y} \mu(d) \chi(d) \eta(d) \sum_{h \leq \frac{y}{d}} \chi(h) \log h,\\
&s_2 = - \sum_{\bfrac{dcr \leq y}{ c \leq U}}\chi(dcr) \mu(d) \Lambda(c) \eta(d),\\
&s_3 = \sum_{n \leq y} \chi(n) \sum_{c > U} \sum_{dc |n} \mu(d) \Lambda(c) (1-\eta(d)). 
\end{split}
\end{equation}
Here in bounding $S_0$ we used Chebychev's estimate 
$$|s_0| \leq \sum_{n \leq U} \Lambda(n) \ll U.$$
In what follows we choose $\eta(\cdot)$ from the paper by Graham, see \cite{Michigan1978}:
$$\eta(d) = \frac{\log \frac{V_0}{d}}{\log \frac{V_0}{V}}, \;\;\; V \leq d \leq V_0.$$
We remind that $\eta(d)=1$ for $d \leq V$ and $\eta(d)=0$ for $d>V_0$.
This choice allows us to win $\log ^{\frac{1}{2}}$ in the last sum, that is of type II.

\subsubsection*{Type I sums}
We start with linear sums among $s_i$ and work with $s_1$ first. 
Write
$$\sum_{h \leq \frac{y}{d}} \chi(h) \log h = \sum_{h \leq \frac{y}{d}} \chi(h) \int_{1}^{h} \frac{du}{u}$$
and exchange the sum and the integral
\begin{equation*}
\begin{split}
s_1& = \sum_{d \leq V_0} \mu(d)\chi(d) \eta(d) \int_{1}^{\frac{y}{d}} \sum_{u \leq h \leq \frac{y}{d}} \chi(h) \frac{du}{u}\\
&= \int_1^y \sum_{d \leq \min (V_0, \frac{y}{u})} \mu(d) \chi(d) \eta(d) \sum_{u \leq h \leq \frac{y}{d}} \chi(h) \frac{du}{u} \\
&\;=\int_1^y \left(\sum_{d \leq V_0} \mu(d) \chi(d) \eta(d) \sum_{u \leq h \leq \frac{y}{d}} \chi(h) \right) \frac{du}{u}\\
&\;=\int_1^y \left(\sum_{d \leq V} \mu(d) \chi(d) \sum_{u \leq h \leq \frac{y}{d}} \chi(h) \right) \frac{du}{u} \\
&+ 
\frac{1}{\log \frac{V_0}{V}}\int_1^y \left(\sum_{d \leq V_0} \mu(d) \chi(d) \log \frac{V_0}{d} \sum_{u \leq h \leq \frac{y}{d}} \chi(h) \right) \frac{du}{u}.
\end{split}
\end{equation*}
Denote the summands $\sigma_1$ and $\sigma_2$. Then
\begin{equation*}
\begin{split}
|\sigma_1| \leq \sum_{d \leq V} \max_{1 \leq u \leq y} \left| \sum_{u \leq h \leq \frac{y}{d}} \chi(h)\right| \int_{1}^{y} \frac{du}{u} \leq (\log y) \sum_{d \leq V} \max_{1 \leq u \leq y} \left| \sum_{u \leq h \leq \frac{y}{d}} \chi(h)\right|.
\end{split}
\end{equation*}
If $q=1$, then we have only trivial $\chi \mmod q$ and
$$|\sigma_1| \leq (\log y) \sum_{d \leq V}\frac{1}{d} \leq x (\log xV)^2.$$
If $q>1$ and $\chi$ is a primitive character  $\mmod q$, we use the P\'{o}lya-Vinogradov inequality(see \cite{Frolenkov2013} for explicit results): for all $x,y$ we have
$$\left| \sum_{x \leq n \leq y} \chi(n)\right| < q^{\frac{1}{2}} \log q.$$
Then 
$$|\sigma_1| < (\log y) \sum_{d \leq V} \max_{1 \leq u \leq y} q^{\frac{1}{2}} \log q < q^{\frac{1}{2}} V (\log xV)^2.$$
Further 
$$|\sigma_2| \leq \frac{\log V_0}{\log \frac{V_0}{V}} |\sigma_1|$$
and
\begin{equation*}
\begin{split}
S_1 &\leq \left(\log \frac{V_0}{V}\right)^{-1} \log \frac{V_0^2}{V} 
\left( \sideset{}{^*}\sum_{\chi \mmod  q=1} \max_{y \leq x} |s_1| +
\sum_{1 <q \leq Q} \frac{q}{\varphi(q)} \sideset{}{^*}\sum_{\chi (q)} \max_{y \leq x} |s_1|\right)\\
&\leq \left(\log \frac{V_0}{V}\right)^{-1} \log \frac{V_0^2}{V} \left(x(\log xV)^2 + V(\log xV)^2 \sum_{1 <q \leq Q} \frac{q^{\frac{3}{2}}}{\varphi(q)} \sideset{}{^*}\sum_{\chi (q)} 1\right)\\
& \leq (x+Q^{\frac{5}{2}}V) (\log xV)^2 \left(\log \frac{V_0}{V}\right)^{-1} \log \frac{V_0^2}{V} .
\end{split}
\end{equation*}

\subsubsection*{Type II sums}
Now we work with $s_2$ and want to use dyadic decomposition. Write
\begin{equation*}
\begin{split}
s_2 &= - \sum_{\bfrac{cdr \leq y}{c \leq U}} \Lambda(c) \mu(d) \eta(d) \chi(cdr) = 
-\sum_{\bfrac{ct \leq y}{c \leq U}} \sum_{d | t} \Lambda(c) \mu(d) \eta(d) \chi(ct) \\
&= -\sum_{c \leq w}-\sum_{w < c \leq U} = s_2' + s_2'',
\end{split}
\end{equation*}
where we introduced a new parameter $w$, that should be smaller than $U$ and will be chosen later.
We deal first with the linear part of $s_2$, namely $s_2'$. Write
$$s_2' = -\sum_{c \leq w} \Lambda(c)\chi(c) \sum_{t \leq \frac{y}{c}} \sum_{d|t} \mu(d) \eta(d) \chi(t).$$
Since we have the bound
$$\left|\sum_{\bfrac{cd=t}{c \leq w}} \Lambda(c) \mu(d) \eta(d) \chi(t)\right| \leq \sum_{c|t}\Lambda(c) = \log t,$$
then proceeding as for $s_1$ via P\'{o}lya-Vinogradov inequality and using the fact that $cd=t \leq wV_0$ we get
$$|S_2'| \leq (x+Q^{\frac{5}{2}}wV_0)(\log (xwV_0))^2,$$
where the $x$ term comes from the contribution of $q=1$ and $Q^{\frac{5}{2}}wV_0$ from the remaining $q \neq 1$.

Next consider $s_2''$.
Writing $s_2''$ as a dyadic sum we have
$$
s_2'' = \sum_{\bfrac{M = 2^{\alpha}}{\frac{1}{2}w < M \leq U}} \sum_{\bfrac{w < c \leq U}{ M < c \leq 2M}} \sum_{{t \leq \frac{y}{c}}} \sum_{d|t}\Lambda(c) \mu(d) \eta(d) \chi(ct).
$$
Using the triangle inequality
$$
S_2'' \leq
\sum_{\bfrac{M = 2^{\alpha} }{ \frac{1}{2}w < M \leq U}}
\sum_{q \leq Q} \frac{q}{\varphi(q)} \sideset{}{^*}\sum_{\chi (q)} \max_{y \leq x}
 \left| \sum_{\bfrac{w < c \leq U}{M < c \leq 2M}}  
 \sum_{t \leq \frac{y}{c}} \sum_{{d|t}} \Lambda(c) \mu(d) \eta(d) \chi(ct) \right|.
$$
By the large sieve inequality we get
\begin{equation*}
\begin{split}
S_2'' \ll \sum_{\bfrac{M = 2^{\alpha} }{ \frac{1}{2}w < M \leq U}} (M'+Q^2)^{\frac{1}{2}}(K'+Q^2)^{\frac{1}{2}} \sigma_1(M)^{\frac{1}{2}} \sigma_2(M)^{\frac{1}{2}} L(M),
\end{split}
\end{equation*}
where 
$M'$ and $K'$ are the number of terms in sums over $c$ and $t$ respectively and
\begin{equation*}
\begin{split}
&\sigma_1(M) = \sum_{\bfrac{w < c <U}{M < c \leq 2M}} \Lambda(c)^{2},\\
&\sigma_2(M)= \sum_{t \leq \frac{y}{M}}\left|\sum_{d|t} \mu(d) \eta(d) \right|^2,\\
\end{split}
\end{equation*}
and
\begin{equation*}
\begin{split}
&L(M) = \log \left( \frac{2x}{M} \min \left(U,2M\right)  \right) \ll \log x,\\
\end{split}
\end{equation*}
By Chebyshev's estimate 
$$ \sigma_1(M) \ll M \log U,$$
then using the estimates $M' \leq M$, $K' \leq \frac{x}{M}$ we have
\begin{equation*}
\begin{split}
&S_2'' \ll (\log x) (\log U)^{\frac{1}{2}} \sum_{\bfrac{M = 2^{\alpha}}{\frac{1}{2}w < M \leq U}}
(M+Q^2)^{\frac{1}{2}} \left(\frac{x}{M}+Q^2\right) ^{\frac{1}{2}} M^{\frac{1}{2}}\sigma_2(M)^{\frac{1}{2}}. 
\end{split}
\end{equation*}
To bound $\sigma_2(M)$ we use a result of Corollary \ref{graham2} and get
\begin{equation*}
\begin{split}
\sigma_2(M) \ll \frac{y}{M \log \frac{V_0}{V}}.
\end{split}
\end{equation*}
Putting it together we obtain
\begin{equation*}
\begin{split}
S_2'' &\ll (\log x) (\log U)^{\frac{1}{2}} x^{\frac{1}{2}}\sum_{\bfrac{M = 2^{\alpha}}{\frac{1}{2}w < M \leq U }} (M+Q^2)^{\frac{1}{2}} \left(\frac{x}{M}+Q^2\right) ^{\frac{1}{2}}  \left(\log \frac{V_0}{V}\right)^{-\frac{1}{2}}\\
& \ll (\log x) 
\frac{(\log U)^\frac{1}{2}}{\left(\log \frac{V_0}{V}\right)^{\frac{1}{2}}} (\log (Uw)) \left(x + \sqrt{2} Qxw^{-\frac{1}{2}} + Q^2 x^{\frac{1}{2}} + U^{\frac{1}{2}}Q x^{\frac{1}{2}}\right),
\end{split}
\end{equation*}
where we applied the bound
$$\sum_{\bfrac{M = 2^{\alpha}}{\frac{1}{2}w < M \leq U }} 1 \leq \frac{\log (2Uw)}{\log 2}.$$
We continue with an estimate for $S_3$ and use of the large sieve inequality (\ref{sieve}) and properties of $\eta(\cdot)$ from Lemma \ref{graham1}. 
Writing $s_3$ as a dyadic sum we have
$$
s_3 = - \sum_{\bfrac{M = 2^{\alpha}}{\frac{1}{2}U < M \leq x/V}} \sum_{\bfrac{U < m \leq x/V}{ M < m \leq 2M}} \sum_{\bfrac{V < k \leq x/M}{ mk \leq y}} \Lambda(m) \left( \sum_{\bfrac{d \mid k}{ d \leq V}} \mu(d) (1-\eta(d)) \right) \chi(mk).
$$
Using the triangle inequality
$$
S_3 \leq
\sum_{\bfrac{M = 2^{\alpha} }{ \frac{1}{2}U < M \leq x/V}}
\sum_{q \leq Q} \frac{q}{\varphi(q)} \sideset{}{^*}\sum_{\chi (q)} \max_{y \leq x}
 \left| \sum_{\substack{U < m \leq x/V \\ M < m \leq 2M}}  
 \sum_{\substack{V < k \leq x/M \\ mk \leq y}} a_m c_k \chi(mk) \right|,
$$
where $a_m = \Lambda(m)$ and $c_k = \sum_{{d | k},\;{d \leq V}} \mu(d) (1-\eta(d))$.
Now apply the large sieve inequality (\ref{sieve}) to get
$$
S_3 \ll
\sum_{\bfrac{M = 2^{\alpha}}{\frac{1}{2}U < M \leq x/V }} 
(M'+Q^2)^{\frac{1}{2}}(K'+Q^2)^{\frac{1}{2}}\sigma_1(M)^{\frac{1}{2}} \sigma_2(M)^{\frac{1}{2}} L(M)
$$
where 
\begin{equation*}
\begin{split}
&\sigma_1(M) = \sum_{\bfrac{V < k \leq x/M }{}} |c_k|^{2},\\
&\sigma_2(M) = \sum_{\bfrac{U < m \leq x/V }{ M < m \leq 2M}}|a_m|^2,\\
\end{split}
\end{equation*}
and
\begin{equation*}
\begin{split}
&L(M) = \log \left( \frac{2x}{M} \min \left(\frac{x}{V},2M\right)  \right) \ll \log x,\\
\end{split}
\end{equation*}
where $M^{\prime}$ and $K^{\prime}$ denote the number of terms in the sums over $m$ and $k$, respectively. 
From the definition of $M'$ and $N'$ we conclude
\begin{equation*}
\begin{split}
& M'=\min\left(2M,\frac{x}{V}\right)-\max\left(M+1,U+1\right) \leq M,\\
& K'=\frac{x}{M}-(V+1)+1 \leq \frac{x}{M}.
\end{split}
\end{equation*}
By Chebyshev's estimate we have an upper bound
$$\sigma_2(M) \leq \sum_{m \leq 2M} \Lambda(m)^2 \leq \psi(2M) \log 2M \ll M\log M.$$
Thus by Cauchy inequality
\begin{equation*}
\begin{split}
&S_3 \ll (\log x) \sum_{\bfrac{M = 2^{\alpha}}{\frac{1}{2}U < M \leq x/V }} (M+Q^2)^{\frac{1}{2}}\left(\frac{x}{M}+Q^2\right)^{\frac{1}{2}}( M \log M)^{\frac{1}{2}} \sigma_1(M)^{\frac{1}{2}}. 
\end{split}
\end{equation*}
Further
$$M(M+Q^2)\left(\frac{x}{M}+Q^2\right)=Mx+Q^2x+M^2Q^2+MQ^4$$
and
$$(\log M)^{\frac{1}{2}} \leq \left( \log \frac{x}{V} \right)^{\frac{1}{2}}.$$
Thus we have
\begin{equation*}
\begin{split}
&S_3 \ll (\log x) \left(\log \frac{x}{V}\right)^{\frac{1}{2}} \sum_{\bfrac{M = 2^{\alpha}}{\frac{1}{2}U < M \leq x/V }} (Mx+Q^2x+M^2Q^2+MQ^4)^{\frac{1}{2}} \sigma_1(M)^{\frac{1}{2}}. 
\end{split}
\end{equation*}

We take $\eta(\cdot)$ from the paper by Graham, see Corollary \ref{graham2} and \cite{Michigan1978}:
$$\eta(d) = 
\begin{cases}
1, &d \leq V,\\
\frac{\log V_0/d}{\log V_0/V}, &V \leq d \leq V_0,\\
0, &d \geq V_0.
\end{cases}
$$ 
so that
$$1-\eta(d) = 1 - \frac{\log \frac{V_0}{d}}{\log \frac{V_0}{V}} = \frac{\log \frac{d}{V}}{\log \frac{V_0}{V}}.$$
On applying Lemma \ref{graham1} we obtain
\begin{equation*}
\begin{split}
\sigma_1(M) & = \sum_{V < k \leq \frac{x}{M}} \left(\sum_{d | k} \mu(d) -\sum_{d|k} \mu(d)\eta(d)\right)^2  \ll \frac{ \log V }{\left(\log \frac{V_0}{V}\right)^2}  \frac{x}{M}
\end{split}
\end{equation*}
that implies
$$S_3 \ll (\log x) \left( \log \frac{x}{V} \right)^{\frac{1}{2}} 
\frac{\left(\log V \right)^{\frac{1}{2}}}{ \log \frac{V_0}{V}} \sum_{\bfrac{M=2^{\alpha}}{\frac{1}{2}U < M \leq \frac{x}{V}}} \left(x^2 + \frac{Q^2 x^2}{M} + MQ^2 x + Q^4 x \right)^{\frac{1}{2}}.$$
Since 
$$\sum_{\bfrac{M = 2^{\alpha}}{\frac{1}{2}U < M \leq x/V }} 1 \leq \frac{\log \frac{2x}{V}}{\log 2},$$
then
$$S_3 \ll \frac{\log x}{\log \frac{V_0}{V}} \left( \log \frac{x}{V} \right)^{\frac{3}{2}} 
\left(\log V \right)^{\frac{1}{2}}
\left(x + \sqrt{2} Qx U^{-\frac{1}{2}}+QxV^{-\frac{1}{2}}+Q^2 x^{\frac{1}{2}}\right).$$
Finally we have to adjust the parameters $U,V,V_0,w$. 
We repeat our previous estimates
\begin{equation*}
\begin{split}
S_0 & \ll UQ^2,\\
S_1 & \leq (x+Q^{\frac{5}{2}}V)(\log xV)^2 \left(\log \frac{V_0}{V}\right)^{-1} \log \frac{V_0^2}{V},\\
S_2' &\leq (x+Q^{\frac{5}{2}}wV_0) (\log (xwV_0))^2,\\
S_2'' & \ll (\log x) \frac{(\log U)^\frac{1}{2}}{\left(\log \frac{V_0}{V}\right)^{\frac{1}{2}}} (\log (2Uw)) (x+\sqrt{2}Qxw^{-\frac{1}{2}}+Q^2x^{\frac{1}{2}}+U^{\frac{1}{2}}Qx^{\frac{1}{2}}),\\
S_3 &\ll \frac{\log x}{\log \frac{V_0}{V}} \left(\log \frac{x}{V}\right)^{\frac{3}{2}} (\log V)^{\frac{1}{2}} (x+\sqrt{2}QxU^{-\frac{1}{2}}+QxV^{-\frac{1}{2}}+Q^2x^{\frac{1}{2}}). 
\end{split}
\end{equation*}
Combining the results above and taking $U=V$ we get
\begin{equation} \label{polylognew}
S=\sum_{q \leq Q} \frac{q}{\varphi(q)} {\sideset{}{^*}{\sum}}_{\chi (q)} \max_{y \leq x} |\psi(y,\chi)| \ll \Pol(x,Q,w,V,V_0) \Log(x,w,V,V_0),
\end{equation}
where
\begin{equation*}
\begin{split}
\Pol(x,Q,w,V,V_0) &= 4x+Q^2V+Q^{\frac{5}{2}}(V+wV_0)+Qx\left(\frac{\sqrt{2}}{w^{\frac{1}{2}}}+\frac{1+\sqrt{2}}{V^{\frac{1}{2}}}\right)
\\
& +2Q^2x^{\frac{1}{2}}+V^{\frac{1}{2}}Qx^{\frac{1}{2}},
\\
\Log (x,Q,w,V,V_0) &= \max \{ (\log xV)^2 \frac{\log \frac{V_0^2}{V}}{\log \frac{V_0}{V}}, (\log (xwV_0))^2, (\log (Vw))\frac{(\log V)^{\frac{1}{2}}}{\left(\log \frac{V_0}{V}\right)^{\frac{1}{2}}} \log x,\\
&\left( \log \frac{2x}{V}\right)^{\frac{3}{2}} \frac{\log 4x}{\log \frac{V_0}{V}} (\log V)^{\frac{1}{2}}\}.
\end{split}
\end{equation*}
Now let's specify $V$ and $V_0$. We introduce a parameter $0<\alpha< \frac{1}{2}$ to be chosen later. We subdivide into two cases 
\begin{enumerate}
\item{$x^{\alpha} \leq Q \leq x^{\frac{1}{2}}$,}
\item{$Q \leq x^{\alpha}$}
\end{enumerate}

and denote $\Pol(x,Q,w,V,V_0)$, $\Log(x,w,V,V_0)$ as $\Pol_1,\Pol_2$ and, respectively $\Log_1$ and $\Log_2$.
If $x^{\alpha} \leq Q \leq x^{\frac{1}{2}}$, then $V = \frac{x^{\beta_1}}{Q}$. We choose $V_0 = \frac{x^{\delta_1}}{Q}$ and $w = \frac{x^{\gamma_1}}{Q}$.
Then putting that into previous expression $\Pol (x,Q,w,V,V_0)$ we get for the factor
\begin{equation*}
\begin{split}
\Pol_1(x,Q)& \ll x+Qx^{\beta_1}+Q^{\frac{3}{2}}x^{\beta_1}+Q^{\frac{1}{2}}x^{\gamma_1+\delta_1}+Q^{\frac{3}{2}}x^{1-\frac{\gamma_1}{2}}\\
&+Q^{\frac{3}{2}}x^{1-\frac{\beta_1}{2}}+Q^2x^{\frac{1}{2}}+Q^{\frac{1}{2}}x^{\frac{1+\beta_1}{2}}.
\end{split}
\end{equation*}
If $Q \leq x^{\alpha}$, we let $V = x^{\beta_2}$, $V_0 = x^{\delta_2}$, $w=x^{\gamma_2}$ and get
\begin{equation*}
\begin{split}
\Pol_2(x,Q)&\ll x+Q^2x^{\beta_2}+Q^{\frac{5}{2}}x^{\beta_2}+Q^{\frac{5}{2}}x^{\gamma_2+\delta_2}+Qx^{1-\frac{\gamma_2}{2}}+ Qx^{1-\frac{\beta_2}{2}}
+Q^2x^{\frac{1}{2}}+Qx^{\frac{1+\beta_2}{2}}.
\end{split}
\end{equation*}
Let $0 < \veps < \frac{1}{14}$. We keep in mind conditions $\alpha < \frac{1}{2}$, $\gamma_1 < \beta_1$, $\delta_1>\beta_1$ and put
\begin{equation*}
\begin{split}
&\alpha=\frac{3}{7}+\veps,\;\; \beta_1=\frac{4}{7}, \;\; \gamma_1=\frac{4}{7}-\veps,\;\; \delta_1=\frac{4}{7}+\frac{5 \veps }{2}.\\
\end{split}
\end{equation*}
Then
\begin{equation*}
\begin{split}
\Pol_1(x,Q)& \ll x+ Qx^{\frac{4}{7}}+Q^{\frac{3}{2}}x^{\frac{4}{7}}+Q^{\frac{1}{2}}x^{\frac{8}{7}+\frac{3 \veps}{2}}+Q^{\frac{3}{2}}x^{\frac{5}{7}+\frac{\veps}{2}} +
Q^{\frac{3}{2}}x^{\frac{5}{7}} + Q^2 x^{\frac{1}{2}}+Q^{\frac{1}{2}}x^{\frac{9}{14}}\\
&\ll x + Q^{2}x^{\frac{1}{2}},
\end{split}
\end{equation*}
where we used 
\begin{equation*}
\begin{split}
&Qx^{\frac{4}{7}} \leq Q^2 x^{\frac{4}{7}-\frac{3}{7}-\veps} <Q^2 x^{\frac{1}{2}},\\
&Q^{\frac{3}{2}}x^{\frac{4}{7}} \leq Q^2 x^{\frac{4}{7}-\frac{3}{14}-\frac{\veps}{2}} < Q^2 x^{\frac{1}{2}},\\
&Q^{\frac{1}{2}}x^{\frac{8}{7}+\frac{3\veps}{2}} \leq Q^2 x^{\frac{8}{7}+\frac{3\veps}{2}-\frac{3}{2}\left(\frac{3}{7}+\veps\right)}=Q^2 x^{\frac{1}{2}},\\
&Q^{\frac{3}{2}} x^{\frac{5}{7}+\frac{\veps}{2}} \leq Q^2 x^{\frac{5}{7}+\frac{\veps}{2}-\frac{3}{14}-\frac{\veps}{2}}=Q^2x^{\frac{1}{2}},\\
&Q^{\frac{3}{2}}x^{\frac{5}{7}} \leq Q^2 x^{\frac{5}{7}-\frac{1}{2} \left( \frac{3}{7}+\veps \right) } = Q^2 x^{\frac{1}{2}-\frac{\veps}{2}} < Q^2 x^{\frac{1}{2}},\\
&Q^{\frac{1}{2}}x^{\frac{9}{14}} \leq Q^2 x^{\frac{9}{14}-\frac{3}{2} \left(\frac{3}{7}+\veps\right)} < Q^2 x^{\frac{1}{2}}.
\end{split}
\end{equation*}
Similarly to satisfy $\gamma_2<\beta_2$, $\delta_2 > \beta_2$  we put
$$\beta_2=\frac{1}{7},\;\;\gamma_2=\frac{1}{7}-\veps,\;\;\delta_2=\frac{1}{7}+\frac{\veps}{2}$$
we obtain
\begin{equation*}
\begin{split}
\Pol_2(x,Q)& \ll x+Q^2x^{\frac{1}{7}}+Q^{\frac{5}{2}}x^{\frac{1}{7}}+Q^{\frac{5}{2}}x^{\frac{2}{7}-\frac{\veps}{2}}+Qx^{\frac{13}{14}+\frac{\veps}{2}}+Qx^{\frac{13}{14}}+Q^2x^{\frac{1}{2}}+Qx^{\frac{4}{7}}\\
&\ll x + Q^2 x^{\frac{1}{2}}+Qx^{\frac{13}{14}+\frac{\veps}{2}},
\end{split}
\end{equation*}
where we used
\begin{equation*}
\begin{split}
&Q^{\frac{5}{2}}x^{\frac{1}{7}} \leq Q^2 x^{\frac{1}{7}+\frac{3}{14}+\frac{\veps}{2}} = Q^2 x^{\frac{5}{14}+\frac{\veps}{2}} < Q^2 x^{\frac{1}{2}},\\
&Q^{\frac{5}{2}}x^{\frac{2}{7}-\frac{\veps}{2}} \leq Q^2 x^{\frac{2}{7}-\frac{\veps}{2}+\frac{3}{14}+\frac{\veps}{2}}=Q^2x^{\frac{1}{2}}.
\end{split}
\end{equation*}

Now we bound $\Log (x,Q,w,V,V_0)$. We notice that with our choice of parameters above $\log \frac{V_0}{V} \gg \log x$, where the implied constant depends on $\beta_i,\delta_i$. Thus $\Log_1 \ll (\log x)^2$ and similarly $\Log_2 \ll (\log x)^2$.
Finally, we have
$$\sum_{q \leq Q} \frac{q}{\varphi(q)} \sideset{}{^*}\sum_{\chi(q)}\max_{y \leq x}|\psi(y,\chi)| \ll (x+Q^2x^{\frac{1}{2}}+Q x^{\frac{13}{14}+\frac{\veps}{2}})(\log x)^{2}.$$
The power $\frac{13}{14} + \frac{\veps}{2}$ is optimal here. Indeed, let us show first that $\alpha > \frac{3}{7}$. The system
$$\begin{cases}
Q^{\frac{1}{2}} x^{\gamma_1+\delta_1} \leq Q^2 x^{\frac{1}{2}},\\
Q^{\frac{3}{2}} x^{1-\frac{\gamma_1}{2}} \leq Q^2 x^{\frac{1}{2}},
\end{cases}
$$
brings us to
$$\begin{cases}
\gamma_1+\delta_1-\frac{3\alpha}{2} \leq \frac{1}{2},\\
1-\frac{\gamma_1}{2}-\frac{\alpha}{2} \leq \frac{1}{2}.
\end{cases}$$
Solving this we obtain $\delta_1 \leq \frac{5\alpha}{2}-\frac{1}{2}$. Further since $Q^{\frac{3}{2}}x^{1-\frac{\beta_1}{2}} \leq Q^2 x^{\frac{1}{2}}$, we get
$$1-\alpha \leq \beta_1 < \delta_1 \leq \frac{5\alpha}{2}-\frac{1}{2}.$$
Thus $\alpha > \frac{3}{7}$. We use that to obtain the fact that the term $Qx^{A}$ has $A > \frac{13}{14}$. Since $Q^{\frac{5}{2}} x^{\gamma_2+\delta_2} \leq Q^2 x^{\frac{1}{2}}$, we get $\gamma_2+\delta_2 \leq \frac{1}{2}-\frac{\alpha}{2} < \frac{2}{7}$. The inequality $Qx^{1-\frac{\beta_2}{2}} \leq Qx^{A}$ gives us $\delta_2 > \beta_2 \geq 2(1-A)$. Similarly for $\gamma_2$ we obtain $\gamma_2 \geq 2(1-A)$ because of the term $Qx^{1-\frac{\gamma_2}{2}} \leq Qx^{A}$. Combining all of this we get
$$4(1-A)<\gamma_2+\delta_2 < \frac{2}{7}$$
and thus $A > \frac{13}{14}$.

\bibliographystyle{abbrv}
\bibliography{../library}{}

\vspace{2em}

\medskip\noindent {\footnotesize Vivatgasse 7, 53111, MPIM, Bonn, Germany
\hfil\break
e-mail: {\tt alisa.sedunova@phystech.edu}}

\end{document}